\documentclass[12pt,reqno]{amsart}
\usepackage[utf8]{inputenc}
\usepackage{amsfonts, amssymb, amsmath, amsthm, mathtools, thmtools}
\usepackage{mathrsfs} 
\usepackage{latexsym}
\usepackage{enumerate}
\usepackage{multicol}
\usepackage{times}
\usepackage{verbatim}
\usepackage{tikz-cd}
\usepackage{fullpage}
\usepackage{here}
\usepackage{url}
\usepackage{csquotes}
\usepackage{placeins}
\usepackage{epic}
\usepackage{graphicx}
\usepackage{epstopdf}
\usepackage{pstricks}
\usepackage{pst-plot}
\usepackage{tikz}
\usetikzlibrary{shapes.geometric, positioning, calc}
\input xypic
\xyoption{all}
\xyoption{poly}
\usepackage[colorlinks=true]{hyperref}
\hypersetup{citecolor=blue, linkcolor=blue}
\usepackage[capitalise]{cleveref}

\crefname{equation}{}{}
\crefname{figure}{{\sc Figure}}{{\sc Figure}}
\crefname{subsection}{Subsection}{Subsections}

\usetikzlibrary{matrix,arrows,decorations.pathmorphing}
\usepackage[euler]{textgreek}
\usepackage{tikz,tikz-cd}
\usepackage[colorinlistoftodos, textwidth = 2.3cm]{todonotes}

\newtheorem{theorem}{Theorem}[section]
\newtheorem{proposition}[theorem]{Proposition}
\newtheorem{lemma}[theorem]{Lemma}
\newtheorem{corollary}[theorem]{Corollary}

\newtheorem{claim}[theorem]{Claim}

\newtheorem*{claim*}{Claim}

\theoremstyle{definition}

\newtheorem{definition}[theorem]{Definition}
\newtheorem{remark}[theorem]{Remark}

\newcommand{\F}{{\mathbb F}}

\newenvironment{poc}{\begin{proof}[Proof of claim]}{\end{proof}}

\usepackage{tabstackengine}
\stackMath

\numberwithin{equation}{section} 
\numberwithin{figure}{section}
\numberwithin{table}{section}

\newcommand{\beq}{\begin{small} \begin{equation}}
\newcommand{\eeq}{\end{equation} \end{small}}

\newcommand{\beqn}{\begin{small} \begin{equation*}}
\newcommand{\eeqn}{\end{equation*} \end{small}}

\title{$f$-Diophantine sets over finite fields via quasi-random hypergraphs from multivariate polynomials}

 \author{Seoyoung Kim}
 \address{Departement Mathematik und Informatik, Universit\"at Basel, Spiegelgasse 1, 4051 Basel, Switzerland}
 \email{seoyoung.kim@unibas.ch}
 \author{Chi Hoi Yip}
 \address{School of Mathematics\\ Georgia Institute of Technology\\ Atlanta, GA 30332\\ United States}
 \email{cyip30@gatech.edu}
 \author{Semin Yoo}
 \address{Discrete Mathematics Group \\ Institute for Basic Science \\ 55 Expo-ro Yuseong-gu, Daejeon 34126 \\ South Korea}
 \email{syoo19@ibs.re.kr}

\begin{document}

\begin{abstract}
We investigate $f$-Diophantine sets over finite fields via new explicit constructions of families of quasi-random hypergraphs from multivariate polynomials. 
In particular, our construction not only offers a systematic method for constructing quasi-random hypergraphs but also provides a unified framework for studying various hypergraphs arising from multivariate polynomials over finite fields, including Paley sum hypergraphs, and hypergraphs derived from Diophantine tuples and their generalizations. 
We derive an asymptotic formula for the number of $k$-Diophantine $m$-tuples, answering a question of Hammonds et al., and study some related questions for $f$-Diophantine sets, extending and improving several recent works. 
We also sharpen a classical estimate of Chung and Graham on even partial octahedra in Paley sum hypergraphs. 
\end{abstract}

\maketitle

\section{Introduction}\label{sec: intro}

Throughout the paper, let $k \geq 2$ be a fixed integer. Let $q$ be an odd prime power, $\F_q$ the finite field with $q$ elements, and $\F_q^*=\F_q \setminus \{0\}$. Let $\overline{\F_q}$ be the algebraic closure of $\F_q$.

A \textit{Diophantine tuple} is a classical object in number theory, consisting of a set of integers (or elements of a given ring) $\{a_{1}, a_{2},\ldots, a_{m}\}$ with the property that the product of any two distinct elements in the set is one less than a square. Fermat is attributed to the first known example $\{1,3,8,120\}$ of integral Diophantine $4$-tuples. Euler extended the set to the rational $5$-tuple $\{1, 3, 8, 120, \frac{777480}{8288641}\}$, and it had been conjectured that there is no \textit{integral} Diophantine $5$-tuple. Finding large Diophantine tuples is related to the counting problem of integral points on elliptic curves: in order to find a positive integer $d$ such that $\{a,b,c,d\}\subset \mathbb{Z}$ is a Diophantine $4$-tuple, we need to solve the following equation for $d$ such that $s,t,r\in \mathbb{Z}$:
\[ad+1=s^2, \quad bd+1=t^2, \quad cd+1=r^2.\]
This is related to the problem of finding an integral point $(d,str)\in \mathbb{Z}\times \mathbb{Z}$ on the elliptic curve
$y^2=(ax+1)(bx+1)(cx+1)$, where $a,b,c \in \mathbb{Z}$.
From this, we can deduce that there are no infinite Diophantine $m$-tuples by Siegel's theorem on the integral points of elliptic curves. On the other hand, Siegel's theorem cannot give an effective upper bound on the size of Diophantine tuples. Similarly, finding a Diophantine tuple of size greater than or equal to $5$ is related to the problem of finding integral points on hyperelliptic curves of genus $g\geq 2$. The conjecture on the non-existence of Diophantine $5$-tuple was confirmed by He, Togb\'e, and Ziegler in \cite{HTZ19}.

The notion of Diophantine tuples has been generalized in various directions, and we refer to a recent book by Dujella \cite{D24} for a comprehensive discussion. In this paper, we focus on the following generalization: Let $k \geq 2$ and $q$ be an odd prime power. We let $f \in \F_q[x_1, \ldots, x_k]$ be a symmetric polynomial. Following~\cite{BDHT16, YY25}, we say $A \subset \F_q$ is an \emph{$f$-Diophantine set over $\F_q$} if $f(a_1, a_2, \ldots, a_k)$ is a square in $\F_q$ whenever $a_1, a_2, \ldots, a_k$ are distinct elements in $A$. In particular, when $k=2$ and $f(x_1,x_2)=x_1x_2+1$, an $f$-Diophantine set over $\F_q$ is exactly a \emph{Diophantine tuple over $\F_q$}, which is of special interest~\cite{21DK, KYY, KYY24b, KYY24a, Shparlinski23}. More generally, when $k \geq 2$ and $f(x_1,x_2,\ldots, x_k)=x_1x_2\cdots x_k+1$, following \cite{k23}, an $f$-Diophantine set with size $m$ over $\F_q$ is said to be a \emph{$k$-Diophantine $m$-tuple over $\F_q$}.  

Important tools for studying Diophantine tuples and their generalizations include Diophantine approximations, Diophantine equations, sieve methods (in particular, Gallagher's larger sieve and its variants; see for example \cite{CY26, KYY, Y24}), elliptic curves, character sum estimates (for example, the Paley graph conjecture in \cite{GM20}), as well as other tools from number theory and arithmetic geometry. For instance, Kazalicki--Naskrecki~\cite{KN} gave a rational parametrization of Diophantine triples by studying the arithmetic of K3 surfaces, and Dujella--Kazalicki \cite{21DK} explored the connection between Diophantine tuples over finite fields and modular forms.

Despite their arithmetic origin, Diophantine tuples have recently been investigated with the help of combinatorial and graph-theoretic methods, which provide powerful tools for understanding their structural properties and quantitative behavior. For example, we refer to Batta--Hajdu--Pongr\'acz \cite{BHP25}, B\'{e}rczes--Dujella--Hajdu--Luca \cite{BDHL11}, Bugeaud--Gyarmati \cite{BG04}, Croot--Yip \cite{CY26}, and Gyarmati--S\'ark\"ozy--Stewart \cite{GSS02} for applications of extremal graph theory to Diophantine tuples and their generalizations over integers, and to Gyarmati \cite{G01} as well as G\"{u}lo\u{g}lu--Murty \cite{GM20} for the finite field setting. These works illustrate how methods from Ramsey theory and Tur\'an theory, can be employed to study Diophantine tuples. In addition, Ramsey theory and spectral graph theory have been successfully applied in the previous work of the authors \cite{KYY24b} to investigate Diophantine tuples over finite fields. Moreover, the authors \cite{KYY24b} as well as Batta--Hajdu--Pongr\'acz \cite{BHP25} studied properties of various notions of ``Diophantine graphs" induced by Diophantine tuples.

In the present paper, we adopt a combinatorial viewpoint and employ tools from polynomials over finite fields to study $f$-Diophantine sets over finite fields. In particular, we focus on the quasi-randomness of ``Diophantine hypergraphs" induced from polynomials over finite fields. We begin by introducing quasi-random hypergraphs and defining the hypergraph $Y_{f,q}$ corresponding to a polynomial $f\in \F_q[x_1,x_2,\ldots, x_k]$. We then state some random-like properties of the hypergraph $Y_{f,q}$ and discuss its applications to the study $f$-Diophantine sets. In particular, we address a question of Hammonds et al. \cite{k23} and improve a result of Chung and Graham \cite{CG90} on Paley sum hypergraphs. These two results correspond to $f$-Diophantine sets with $f=x_1x_2\cdots x_k+1$ and $f=x_1+x_2+\cdots+x_k$, respectively, and they extend to a larger class of polynomials $f$. 

\subsection{Quasi-random hypergraphs}
The concept of quasi-randomness for graphs was introduced in \cite{CGW89}, and independently in \cite{T87b,T87a}.
In particular, Chung, Graham, and Wilson \cite{CGW89} showed that several
seemingly unrelated properties of quasi-random graphs are, in fact, equivalent.
Building on this, Chung and Graham \cite{C90, CG90} extended the concept of quasi-randomness to uniform hypergraphs, maintaining a similar perspective as for quasi-random graphs. 
In particular, they showed that when the edge density of a family of hypergraphs is $1/2$, several analogous statements of the properties of quasi-random graphs are asymptotically equivalent.
Since the work of Chung and Graham, various approaches to generalizing the concept of quasi-randomness for hypergraphs have been explored in the literature; see for example \cite{C90, CHPS12, HT89}, addressing different directions and applications.
For the purpose of our paper, we adopt the notion of quasi-randomness defined by Chung and Graham \cite{CG90} as their definitions include subgraph counting for general hypergraphs. 

We first introduce some necessary definitions closely following the presentation in \cite{CG90}. We say that a hypergraph is \emph{$k$-uniform} if each hyperedge contains precisely $k$ vertices.
Let $G$ and $H$ be $k$-uniform hypergraphs.
We use $N^*_G(H)$ to denote the number of labeled occurrences of $H$ as an induced subgraph of $G$.
For a family $\mathcal{F}$ of hypergraphs, we denote $N_G^*(\mathcal{F})=\sum_{H\in \mathcal{F}}N_G^*(H)$. 
Next, we define a special type of $k$-uniform hypergraph that plays a key role in defining quasi-random hypergraphs.
A $k$-uniform hypergraph is called a \textit{$k$-octahedron}, denoted by 
\[\mathcal{O}_k=\mathcal{O}_k(u_1(0),u_1(1),u_2(0),u_2(1),\ldots,u_k(0),u_k(1)),\]
where the vertex set consists of $u_i(\varepsilon_i)$ such that $1 \le i \le k$ and $\varepsilon_i \in \{0,1\}$.
The edge set is defined as the set of all $k$-sets of the form $\{u_1(\varepsilon_1),u_2(\varepsilon_2),\ldots,u_k(\varepsilon_k) \mid \varepsilon_i \in \{0,1\},1\le i\le k\}$.
In other words, a $k$-octahedron is a complete $k$-uniform $k$-partite hypergraph where each vertex class has two vertices.
Further, we say a $k$-uniform hypergraph is an \textit{even partial octahedron} (EPO) if it has the same vertex set as $\mathcal{O}_k$, and the intersection of its edge set and the edge set of $\mathcal{O}_k$ has even size.  
The set of all EPOs defined on $\mathcal{O}_k$ is denoted by $\mathcal{O}_k^e$. 

Following Chung and Graham~\cite{CG90}, we introduce the following properties of $k$-uniform hypergraphs. 
\begin{definition}\cite[Section 3]{CG90}  \label{def: quasirandom}
Let $\mathcal{G}$ be a family of $k$-uniform hypergraphs with an unbounded number of vertices; we say $\mathcal{G}$ has a property below if all $G \in \mathcal{G}$ satisfy the property.   
\begin{enumerate}
    \item[$Q_1(s)$ :] For all $k$-uniform hypergraphs $G'$ with $s$ vertices,
    \[N^*_{G}(G')=(1+o(1))\frac{n^s}{2^{\binom{s}{k}}}, \quad \text{ as } |V(G)|=n \to \infty.\]
    \item[$Q_3$ :] The number of induced copies of EPOs in $G$ is 
    \begin{equation*}
        N^*_G(\mathcal{O}_k^e) \le (1+o(1))\frac{n^{2k}}{2}, \quad \text{ as } |V(G)|=n \to \infty.
    \end{equation*}
\end{enumerate}
\end{definition}
The main result in \cite{CG90} states that if $s\geq 2k$, then $Q_1(s)$ and $Q_3$ are equivalent \footnote{On the other hand, as remarked in \cite[Section 11]{CG90}, $Q_1(s)$ is not equivalent to $Q_3$ for $s\leq k+1$. }. Following \cite{CG90}, we say $\mathcal{G}$ is a family of \textit{quasi-random ($k$-uniform) hypergraphs} if $\mathcal{G}$ has property $Q_3$. For the purpose of this paper, we also introduce the following two additional equivalent properties, $Q_1$ and $Q_3'$, that can be used to define quasi-random hypergraphs.
\begin{itemize}
    \item[$Q_1$ :] $Q_1(s)$ holds for all positive integers $s$. 
    \item[$Q_3'$ :] The number of induced copies of EPOs in $G$ is 
    \[N^*_G(\mathcal{O}_k^e)= (1+o(1))\frac{n^{2k}}{2}, \quad \text{ as } |V(G)|=n \to \infty.\]
\end{itemize}

It is easy to show that for each positive integer $s$, $Q_1(s+1)$ implies $Q_1(s)$ \cite[page 108]{CG90}. It follows that $Q_1$ and $Q_3$ are equivalent, since, as mentioned earlier, $Q_3$ is equivalent to $Q_1(s)$ for each $s\geq 2k$. The equivalence of $Q_3$ and $Q_3'$ can be found in \cite[Fact 2]{CG90}. 

Despite the importance of the quasi-randomness of hypergraphs in combinatorics, explicit families of quasi-random hypergraphs remain largely unexplored. One of the classical examples of quasi-random hypergraphs is the family of Paley (sum) hypergraphs; see \cite[Section 10]{CG90}.  For an odd prime power $q$, the \textit{Paley (sum) $k$-graph} $P_k(q)$ is the $k$-uniform hypergraph whose vertices are the elements of $\F_q$, such that a $k$-set $\{u_1,u_2,\ldots,u_k\}$ forms an edge if and only if $u_1+u_2+\cdots+u_k$ is a square in $\F_q$.
Other than Paley $k$-graphs and even intersection $k$-graphs in \cite[Section 10]{CG90}, we refer to \cite{AH20} for other examples of such families.

The main goal of this paper is to provide a new systematic approach to constructing explicit families of quasi-random hypergraphs using multivariate polynomials over finite fields.
By using the quasi-randomness of these new constructions, we also resolve a question posed in \cite[Section 4]{k23} related to Diophantine equations. Specifically, we derive an asymptotic formula for the number of $k$-Diophantine $m$-tuples over finite fields; see \cref{sec:app} for a more precise statement.

\subsection{Main results}
Let $q$ be an odd prime power and let $f\in \mathbb{F}_q[x_1,x_2,\ldots,x_k]$ be a symmetric polynomial. We define the $k$-uniform hypergraph $Y_{f,q}$ on $\F_q$, such that a $k$-element subset $\{a_1,a_2\ldots,a_k\}$ of $\F_q$ forms an edge if and only if $f(a_1,a_2,\ldots,a_k)$ is a square in $\F_q$.

\begin{definition}\label{def:admissible}
Let $f \in \F_{q}[x_1,x_2,\ldots,x_k]$ be a symmetric polynomial with degree $d$. We say that $f$ is \emph{admissible} if it satisfies the following two conditions:
\begin{itemize}
    \item $f(x_1,x_2,\ldots,x_k)$ is not a constant multiple of the square of a polynomial in $\F_q[x_1,x_2,\ldots,x_k]$.  
    \item $f$ satisfies the \emph{primitive condition}: when we write $$f(x_1,x_2,\ldots , x_k)=\sum_{j=0}^{d}H_{j}(x_2,x_3,\ldots, x_{k})x_1^j,$$
   the polynomials $H_0,H_1, H_2, \ldots, H_d$ do not share a common zero over $\overline{\F_q}$.
\end{itemize}
\end{definition}

We now present our main result.
\begin{theorem}\label{thm:main}
Let $k\geq 2$ and let $d\geq 1$. Let $\mathcal{Y}_{d}$ be the family consisting of all hypergraphs $Y_{f,q}$, where $q$ is an odd prime power and $f \in \F_q[x_1,x_2,\ldots,x_k]$ is an admissible polynomial with degree $d$.
Then $$N^*_{Y_{f,q}}(\mathcal{O}_k^e)=\frac{q^{2k}}{2}+O(q^{2k-1})$$
holds uniformly for all $Y_{f,q} \in \mathcal{Y}_{d}$. In particular, $\mathcal{Y}_{d}$ is a family of quasi-random $k$-uniform hypergraphs.
\end{theorem}

As a direct consequence of \cref{thm:main}, we have the following corollary.
\begin{corollary}\label{cor: Paley}
The number of induced copies of EPOs in $P_k(q)$ is $q^{2k}/2+O(q^{2k-1}).$
In particular, the family of Paley $k$-graphs is quasi-random.
\end{corollary}

In particular, \cref{cor: Paley} improves the error term in the previous result $N_{P_k(q)}^*(\mathcal{O}_k^e)= q^{2k}/2+O(q^{2k-1/2})$
proved by Chung and Graham \cite[Section 10]{CG90}.

In \cref{sec:background}, we provide further motivation for the two conditions used to define admissible polynomials. These two conditions are fundamental. Indeed, if the first condition is not satisfied, then the graph $Y_{f,q}$ would be far from being quasi-random; see \cref{rem:far}. While the second condition seems restrictive, it is not hard to show that if $d$ and $k$ are fixed, then almost all symmetric polynomials in $\F_q[x_1,x_2,\ldots, x_k]$ with degree $d$ are admissible, as $q\to \infty$; see \cref{rem:almostall}. Thus, informally speaking, almost all $Y_{f,q}$ are quasi-random. 
In addition, we compare our second condition with the notion ``primitive kernel" for $2$-variable polynomials introduced by Gyarmati and S\'{a}rk\"{o}zy \cite{GS08}; see \cref{rem:primitivekernel}.

A key ingredient in our proof of \cref{thm:main} is the character sum computation carried out in the
proof of \cref{prop:key}.

\subsection{Applications to Diophantine tuples over finite fields}\label{sec:app}

Finally, we discuss some applications of \cref{thm:main} to Diophantine tuples. \cref{thm:main} implies the following corollary immediately.

\begin{corollary}\label{cor:Dq}
The family of $k$-graphs induced by $k$-Diophantine tuples is quasi-random. More precisely, if $k\geq 2$, then the family of $k$-graphs $\mathcal{D}_{k}=\{Y_{f,q}: q \text{ is an odd prime power}\}$ is quasi-random, where $f(x_1,x_2,\ldots, x_k)=x_1x_2\cdots x_k+1$.
\end{corollary}

In \cite[Section 4]{k23}, the authors asked the following question: Can we count the number of $k$-Diophantine $m$-tuples in $\F_q$ for a given odd prime power $q$? While an explicit formula is believed to be out of reach in general, an asymptotic formula for the case $k=m$ was answered in \cite[Theorem 1.7]{k23}, and the general case $m>k$ was left as an open problem. 

We resolve their question in the following corollary. 
\cref{cor:Dq} implies that the family $\mathcal{D}_k$ has the property $Q_1$, so the corollary follows immediately.

\begin{corollary}\label{cor:mtuple}
Let $k$ and $m$ be fixed positive integers with $2\leq k\leq m$. Let $q$ be an odd prime power. Then as $q \to \infty$, the number of $k$-Diophantine $m$-tuples in $\F_q$ is  
$
q^m/\big(m! \cdot 2^{\binom{m}{k}}\big)+o(q^m).
$
\end{corollary}

Note that when $m=k$, the above corollary recovers \cite[Theorem 1.7]{k23}. We remark that the case $k=2$ in \cref{cor:mtuple} has been proved by Dujella and Kazalicki \cite[Section 6]{21DK}; however, as remarked in \cite[Section 4]{k23}, it appears that the techniques in \cite{21DK, k23} do not extend to $k$-Diophantine tuples with $k \geq 3$ and $m>k$. 

Our proof of \cref{thm:main} (and thus \cref{cor:mtuple}) only relies on Weil's bound on complete character sums (with single variable polynomial argument, see \cref{Weil}) and the main result in Chung and Graham~\cite{CG90}. However, since we are dealing with multivariate polynomials and we already have an asymptotic formula, naturally one expects that there is a power saving in the error term by appealing to more sophisticated higher dimensional versions of Weil's bound. Indeed, using a consequence of the Lang-Weil bound for geometrically
irreducible varieties \cite{LW54}, in \cref{Sec5}, we obtain the following effective version of \cref{cor:mtuple} that holds for all admissible polynomials $f$.

\begin{theorem}\label{thm:asymp}
Let $k\geq 2$ be a fixed integer. Then uniformly for all odd prime powers $q$, all admissible polynomials $f \in \F_q[x_1,x_2,\ldots,x_k]$ with degree $d\geq 1$ and all integers $m\geq k$, the number of $f$-Diophantine sets of size $m$ is 
$$
\frac{q^m}{m! \cdot 2^{\binom{m}{k}}}+O\big((2d)^{2\binom{m}{k}} q^{m-1/2}+(2d)^{13\binom{m}{k}/3}q^{m-1}\big).
$$
\end{theorem}

We remark that Shparlinski~\cite{Shparlinski23} showed that the error term in \cref{thm:asymp} can be improved to $O_m(q^{m-1})$ when $k=2$ and $f(x_1,x_2)=x_1x_2+r$ for $r\in \F_q^*$. In addition, very recently, Gu \cite{Gu26} extended Shparlinski's method to polynomials in several variables of the form $f(x_1,x_2,\ldots, x_k)=x_1x_2\cdots x_k+r$, obtaining the same error term. 
We also note that the main result of \cite{Gu26} establishes a special case of \cref{thm:asymp}, namely for admissible polynomials of the form $g(x_1x_2\cdots x_k)$, where $g$ is a univariate polynomial. Gu further remarked that his method does not appear to extend to general admissible polynomials.

It is well-known that a random $k$-graph $G_k(n, 1/2)$ has clique number, the number of vertices of the largest induced complete subgraph of $G_k(n, 1/2)$, of the order $(\log n)^{1/(k-1)}$ almost surely. In view of \cref{thm:main}, whenever $f \in \F_q[x_1,x_2,\ldots, x_k]$ is admissible with degree $d$, since the set of squares in $\F_q$ ``behaves like" a random subset of $\F_q$ with density $1/2$, one expects that $\omega(Y_{f,q})$, the clique number of $Y_{f,q}$, equivalently, the largest size of $f$-Diophantine sets over $\F_q$, to be at least $c_{d,k}(\log q)^{1/(k-1)}$, where $c_{d,k}$ is a positive constant depending only on $d$ and $k$. When $k=2$, this heuristic has been confirmed in \cite[Theorem 1.4 and Theorem 1.6]{KYY24b} by introducing a quantitative way to measure the quasi-randomness for graphs. However, studying cliques in hypergraphs is usually much harder compared to that in graphs, and it seems challenging to extend the method in \cite{KYY24b} to hypergraphs. 

\cref{thm:asymp} readily provides the following lower bound on the clique number of $Y_{f,q}$, which, while not necessarily optimal, represents an improvement on some previous results.

\begin{corollary}
If $k\geq 2$, $q$ is an odd prime power, and $f \in \F_q[x_1,x_2,\ldots,x_k]$ is an admissible polynomial with degree $d\geq 1$, then $\omega(Y_{f,q})\geq c_k(\log_{\widetilde{d}} q)^{1/k}$, where $\widetilde{d}=\max(2,d)$, and $c_k$ is a positive constant only depending on $k$. 
\end{corollary}

We remark that if $f$ is a sparse polynomial, a better lower bound on $\omega(Y_{f,q})$ was proved in \cite{YY25} based on some explicit algebraic constructions. More precisely, under some additional structural assumptions on $f$, it is proved in \cite[Theorem 1.1]{YY25} that $\omega(Y_{f,q})\geq (1-o(1))\frac{1}{d}(\log_4 q)^{1/\ell(f)}$, where $\ell(f)$ is the number of non-constant monomials in the monomial expansion of $f$ and $d$ is the degree of $f$. In particular, $\omega(Y_{f,q})\geq (\frac{1}{k}-o(1))\log_4 q$ when $f=x_1x_2\ldots x_k+1$. However, note that the parameter $\ell(f)$ is much larger compared to $k$ for a generic polynomial $f$, and thus the above corollary refines this result for almost all admissible polynomials.

\medskip

\textbf{Notations.} We follow the standard asymptotic notations.  We use the standard big $O$ notation, meaning the implicit constant is absolute. $d$ always stands for the degree of a polynomial, and $O_d$ means the implicit constant only depends on the degree $d$ (but not on the polynomial itself).

\section{More backgrounds and motivations}\label{sec:background}
In this section, we give more backgrounds and motivations for defining admissible polynomials (\cref{def:admissible}). 

We first recall the following fundamental result, which is a special case of the Schwartz–Zippel lemma \cite[Corollary 1]{S80} applied to finite fields. We will apply the lemma repeatedly in the paper.
\begin{lemma}[Schwartz–Zippel lemma] \label{SZ}
Let $d\geq 1$ and let $g\in \overline{\F_q}[x_{1},x_{2},\ldots,x_{m}]$ be a polynomial with degree $d$. Then the number of $(u_1, u_2, \ldots, u_m)\in \F_q^m$ such that $g(u_{1},u_{2},\ldots ,u_{m})=0$ is at most $dq^{m-1}$.
\end{lemma}

\begin{remark}\label{rem:far}
Assume that $f(x_1,x_2, \ldots, x_k)$ is a constant multiple of the square of a polynomial. 
In this case, $Y_{f,q}$ is far from being quasi-random. Say $f=cg^2$, where $c\in \F_q$ and $g \in \F_q[x_1,x_2,\ldots, x_k]$. If $c$ is a square in $\F_q$, then by definition, $Y_{f,q}$ is complete. On the other hand, if $c$ is a non-square in $\F_q$, then $k$ distinct vertices $\{v_1,v_2,\ldots, v_k\}$ form an edge in $Y_{f,q}$ if and only if $f(v_1,v_2, \ldots, v_k)=0$. Thus, by the Schwartz–Zippel lemma, the number of edges of $Y_{f,q}$ is $O_d(q^{k-1})$ and thus $Y_{f,q}$ is sparse.
\end{remark}

Next, we give two remarks related to the primitive condition in the definition of admissible polynomials.
\begin{remark}\label{rem:primitivekernel}

     A polynomial $f(x,y) \in \F_q[x,y]$ can be expressed as
     $$
     f(x,y)=\sum_{i=0}^m G_i(y)x^i=\sum_{j=0}^n H_j(x)y^j.
     $$
     Following \cite{GS08}, we say $f$ is \emph{primitive in $x$} if $\gcd(G_0(y),\ldots, G_m(y))=1$, and \emph{primitive in $y$} if $\gcd(H_0(x),\ldots, H_n(x))=1$.     
     In \cite{GS08}, Gyarmati and S\'{a}rk\"{o}zy proved that any two-variable polynomial $f(x,y)\in \mathbb{F}_q[x,y]$ has a factorization, which is unique up to constant factors, such that $f(x,y)=F(x)G(y)H(x,y)$, where $H(x,y)$ is the \emph{primitive kernel} of $f(x,y)$ in the sense that $H(x,y)$ is primitive in both $x$ and $y$.

     Next, we show that our primitive condition in \cref{def:admissible} extends their primitivity notion to polynomials with $k\geq 2$ variables. Recall we say that $f \in \F_q[x_1,x_2,\ldots, x_k]$ satisfies the primitive condition provided that when we write $f(x_1,x_2,\ldots , x_k)=\sum_{j=0}^{d}H_{j}(x_2,x_3,\ldots, x_{k})x_1^j,$
   the polynomials $H_0,H_1, H_2, \ldots, H_d$ do not share a common zero over $\overline{\F_q}$. Since we work on symmetric polynomials, the same condition holds if we expand $f$ based on each $x_i$. Note that if $k=2$, this primitive condition is equivalent to $\gcd(H_0, H_1, \ldots, H_d)=1$, which is consistent with the above notion of primitivity of $2$-variable polynomials. 
   
   However, in general, when $k\geq 3$, the condition $\gcd(H_0, H_1, \ldots, H_d)=1$ is not equivalent to the condition that $H_0,H_1, H_2, \ldots, H_d$ do not share a common zero over $\overline{\F_q}$. For example, if $f(x_1,x_2,x_3)=x_1x_2+x_2x_3+x_3x_1=H_1(x_2,x_3)x_1+H_2(x_2,x_3)$ with $H_1(x_2,x_3)=x_2+x_3$ and $H_2(x_2,x_3)=x_2x_3$, then $\gcd(H_1, H_2)=1$, however, $H_1$ and $H_2$ share the common root $x_2=x_3=0$. The same example also shows that the factorization result mentioned above does not extend in general to the $3$ or more variables case: the polynomial $f(x_1,x_2,x_3)$ does not satisfy the primitive condition, however, it is irreducible over $\mathbb{F}_q$.

\end{remark}
\begin{remark}\label{rem:almostall}
    One can prove that almost all multivariable symmetric polynomials of any degree are admissible over $\mathbb{F}_q$. We briefly illustrate this for the case of $3$-variable polynomials of degree $2$. We can write any symmetric $f(x,y,z)$ of degree $2$ as 
\beq
\label{deg2-var3-case}
f(x,y,z)=Ax^2+Ay^2+Az^2+Bxy+Byz+Bxz+Cx+Cy+Cz+D.
\eeq 
We let $F_0=F_0(y,z)=D+Cy+Cz+Byz+Ay^2+Az^2$ and $F_1=F_1(y,z)=By+Bz+C$ so that
\[f(x,y,z)=Ax^2+F_1(y,z)x+F_0(y,z).\]
 When $A\neq 0$, we can guarantee that $f(x,y,z)$ satisfies the primitive condition. Now we assume $A=0$. Then we have 
    $f(x,y,z)=F_0(y,z)+F_1(y,z)x$,
    and we just need to determine the cases when $F_0$ and $F_1$ share common zeros or not to determine the primitivity of $f(x,y,z)$. That is to say, whether their resultant is zero or not. The resultant of $F_0$ and $F_1$ in terms of the variable $y$ is $$\textrm{Res}(F_0,F_1,y)=\det \begin{bmatrix}
Bz+C & B \\
D+Cz & C+Bz 
\end{bmatrix}  = B^2z^2+BCz+C^2-BD.$$
Hence $\textrm{Res}(F_0,F_1,y)$ has no roots in $\overline{\F_q}$ if and only if $B=0$ and $C^2-BD\neq 0$, thus if and only if $B=0$, and $C\neq 0$ (If $C=0,D=0$, then the resultant is zero when $z=0$ for any $B$. This is the case of $xy+yz+zx$ in \cref{rem:primitivekernel}.). Hence, the probability of a randomly chosen polynomial \cref{deg2-var3-case} satisfying the primitive condition is
$$
\frac{(q-1)q^3 + (q-1)q}{q^4} \to 1 \quad \text{as  $\quad q\to \infty$.}
$$
On the other hand, it is easy to show that the probability of a randomly chosen polynomial \cref{deg2-var3-case} being a constant multiple of the square of a polynomial in $\F_q[x,y,z]$ tends to $0$ as $q\to \infty$.

The same argument shows that if $d$ and $k$ are fixed, then almost all symmetric polynomials in $\F_q[x_1,x_2,\ldots, x_k]$ with degree $d$ are admissible, as $q\to \infty$. Indeed, to see that almost all such polynomials satisfy the primitive condition, it suffices to consider the coefficient for $x_1^d$.
\end{remark}

\section{A key proposition}\label{sec:keyprop}
The main purpose of this section is to prove \cref{prop:Y}, which will play a key role in \cref{Sec4}.
Our proof adapts the approach of Gyarmati and S\'{a}rk\"{o}zy \cite[Section 6]{GS08}, extending their results from the $2$-variable case to the $k$-variable case for any $k\geq 2$. 
\begin{proposition}\label{prop:Y}
Let $q$ be odd prime power and $k\geq 2$. Assume that $f(x_1,x_2, \ldots, x_k)\in \F_q[x_1,x_2,\ldots, x_k]$ is admissible with degree $d$. 
Denote by $X$ the set of tuples $( u_{2},u_{3},\ldots,u_{k})\in \F_q^{k-1}$ with $f(x_1,u_{2},u_{3},\ldots,u_{k})=ch^2(x_1)$ for some $c\in \F_q$ and monic $ h(x_1) \in \F_q[x_1]$.
Then we have 
$|X|\leq (d^2+d)q^{k-2}.$
\end{proposition}

Write  
\begin{equation}\label{eq:f}
f(x_1,x_2,\ldots,x_k)=p_n(x_2,\ldots,x_k)x_1^n+\cdots+p_1(x_2,\ldots,x_k)x_1+p_0(x_2,\ldots,x_k),    
\end{equation} 
where $p_n$ is not identically zero.
Note that $n \geq 1$; otherwise, $f(x_1,x_2, \ldots, x_k)$ is a constant polynomial, violating the assumption. We define 
\begin{align*}
Y&=\{(u_{2},\ldots,u_{k}) \in X \mid p_n(u_{2},\ldots,u_{k})=0\}, \quad \text{and} \\
Z&=\{ (u_{2},\ldots,u_{k}) \in 
X \mid p_n(u_{2},\ldots,u_{k}) \ne 0\}.
\end{align*}
Then $|X|=|Y|+|Z|$.
By the Schwartz–Zippel lemma, we have $|Y| \le (d-n)q^{k-2}\le dq^{k-2}$. 
We now proceed to estimate $|Z|$.
The following lemma completes the proof of \cref{prop:Y}.

\begin{lemma} \label{bound-Y4}
   We have $|Z|\leq ndq^{k-2}.$ In particular, $|Z|\leq d^2q^{k-2}$.
\end{lemma}

\begin{proof}
If $n$ is odd, then clearly $Z=\emptyset$ and we are done. Next, assume that $n$ is even. 

If $(u_{2,1},u_{3,1},\ldots,u_{k,1}),\ldots ,(u_{2,T},u_{3,T},\ldots,u_{k,T})$ are all the distinct elements of $Z$, then for each $i=1,\ldots , T$, the polynomials $f(x_1,u_{2,i},u_{3,i},\ldots,u_{k,i})$ can be written in the form of 
\[f(x_1,u_{2,i},u_{3,i},\ldots,u_{k,i})=C_ih^2_i(x_1),\quad \text{with $C_i\in\F_q$ and monic $h_i(x_1)\in \F_q[x_1].$}\]
Note that $C_i=p_n(u_{2,i},u_{3,i},\ldots,u_{k,i})$ and hence 
\begin{equation}\label{eq:fph}
f(x_1,u_{2,i},u_{3,i},\ldots,u_{k,i})=p_n(u_{2,i},u_{3,i},\ldots,u_{k,i})h^2_i(x_1).    
\end{equation}
We denote 
\begin{equation}\label{eq:h_i}
h_i(x_1)=a_r(i)x_1^r+a_{r-1}(i)x_1^{r-1}+\cdots + a_1(i)x_1 + a_0(i),
\end{equation}
where $a_r(i)=1$, $a_{r-j}(i)\in \F_q$, and $r=\frac{n}{2}$. 

\begin{claim}\label{claim1}
For each $0\leq j \leq r$, there exists a polynomial $q_{r-j}(x_2,x_3,\ldots,x_k)\in \F_q[x_2,x_3,\ldots,x_k]$ with degree at most $jd$, such that 
\begin{equation}\label{eq:ar-j}
    a_{r-j}(i)(p_n(u_{2,i},u_{3,i},\ldots,u_{k,i}))^j=q_{r-j}(u_{2,i},u_{3,i},\ldots,u_{k,i})
\end{equation}
holds for $1\leq i \leq T$. 
\end{claim}

\begin{poc}
We prove the claim by induction. The case $j=0$ is obvious since we can simply set $q_{r} \equiv 1$. Let $1\leq j\leq r$ and assume that the claim holds for $0,1, \ldots j-1$. The coefficient of $x_1^{n-j}$ in $p_n(u_{2,i},u_{3,i},\ldots,u_{k,i})^j h^2_{i}(x_1)$ is
\begin{small}
\begin{align*}
&(p_n (u_{2,i},\ldots,u_{k,i}))^j\left(2a_{r-j}(i) +\sum_{\substack{0\leq t_0,\ldots t_{j-1} \\ t_0+\cdots +t_{j-1}=2 \\ t_1+2t_2+\cdots +(j-1)t_{j-1}=j}} \frac{2}{t_0! \ldots t_{j-1}! } a_{r-1}(i)^{t_1}\cdots a_{r-(j-1)}(i)^{t_{j-1}}\right)\\
&=2 a_{r-j}(i) (p_n(u_{2,i},\ldots,u_{k,i}))^j \\
&+\sum_{\substack{0\leq t_0,\ldots t_{j-1} \\ t_0+\cdots +t_{j-1}=2 \\ t_1+2t_2+\cdots +(j-1)t_{j-1}=j}} \frac{2}{t_0! \ldots t_{j-1}! } (a_{r-1}(i)p_n(u_{2,i},\ldots,u_{k,i}))^{t_1}\cdots (a_{r-(j-1)}(i)p_n(u_{2,i},\ldots,u_{k,i})^{j-1})^{t_{j-1}}, 
\end{align*}
\end{small}
which, by the inductive hypothesis, is of the form
\[2 a_{r-j}(i) (p_n(u_{2,i},u_{3,i},\ldots,u_{k,i}))^j+s_j(u_{2,i},u_{3,i},\ldots,u_{k,i}),\]
where $s_j(x_2,x_3,\ldots, x_k)\in \F_q[x_2,x_3,\ldots, x_k]$ is a polynomial (independent of $i$) of degree $\leq jd$. Multiplying $(p_n(u_{2,i},u_{3,i},\ldots,u_{k,i}))^{j-1}$ on both sides of equation~\eqref{eq:fph} and comparing the coefficient of $x_1^{n-j}$ , we have the equality 
\begin{align*}
p_{n-j}(u_{2,i},\ldots,u_{k,i})(p_n(u_{2,i},\ldots,u_{k,i}))^{j-1}
=2 a_{r-j}(i) (p_n(u_{2,i},\ldots,u_{k,i}))^j+s_j(u_{2,i},\ldots,u_{k,i}),  
\end{align*}
which implies
\beqn
a_{r-j}(i) (p_n(u_{2,i},\ldots,u_{k,i}))^j =\frac{1}{2}(p_{n-j}(u_{2,i},\ldots,u_{k,i})(p_n(u_{2,i},\ldots,u_{k,i}))^{j-1}-s_j(u_{2,i},\ldots,u_{k,i})).  
\eeqn
Thus, we can set $$q_{r-j}(x_2,\ldots,x_k)=\frac{1}{2}\big(p_{n-j}(x_2,\ldots,x_k)(p_n(x_2,\ldots,x_k))^{j-1}-s_j(x_2,\ldots,x_k)\big),$$
with degree at most $dj$ such that equation~\eqref{eq:ar-j} holds, as required.
\end{poc}
We define
\begin{align*}
h & (x_1,x_2,\ldots, x_k)\\ & =(p_n(x_2,\ldots, x_k))^r x_1^r +q_{r-1}(x_2,\ldots, x_k)(p_n(x_2,\ldots, x_k))^{r-1}x_1^{r-1}+  \cdots + q_0(x_2,\ldots, x_k),
\end{align*}
where the polynomials $q_{r-j}$'s are from \cref{claim1} and the degree of $q_{r-j}$ is at most $jd$. Then by equations \eqref{eq:h_i} and 
\eqref{eq:ar-j}, we have
$$
(p_n(u_{2,i},\ldots,u_{k,i}))^r h_i(x_1)= h(x_1,u_{2,i},\ldots,u_{k,i})
$$
for $1\leq i \leq T$. 
It then follows from equation~\eqref{eq:fph} that 
\begin{equation}\label{eq:fph2}
(p_n(u_{2,i},\ldots,u_{k,i}))^{n-1}f(x_1,u_{2,i},\ldots,u_{k,i})=h^2(x_1,u_{2,i},\ldots,u_{k,i})  
\end{equation}
holds for each $1\leq i \leq T$. 

Let $1\le \ell \le n$. From the definition of $h(x_1,x_2,\ldots, x_k)$ and~\eqref{eq:fph2}, we can deduce that  the coefficient of $x_1^{n-\ell}$ of $(p_n(u_{2,i},\ldots,u_{k,i}))^{n-1}f(x_1,u_{2,i},\ldots,u_{k,i})$  is 
\[r_{n-\ell}(u_{2,i},\ldots,u_{k,i}) (p_n(u_{2,i},\ldots,u_{k,i}))^{n-\ell},\]
where $r_{n-\ell}(x_2,\ldots,x_k) \in \F_q[x_2,\ldots, x_k]$ is a polynomial of degree at most $\ell d$. 

Multiplying both sides of equation~\eqref{eq:f} by $(p_n(x_2,\ldots,x_k))^{n-1}$ and comparing the coefficient of $x_1^{n-\ell}$, for each $1\leq i \leq T$, we deduce from equation~\eqref{eq:fph2} that 
\beqn p_{n-\ell}(u_{2,i},\ldots,u_{k,i})(p_n(u_{2,i},\ldots,u_{k,i}))^{n-1}=r_{n-\ell}(u_{2,i},\ldots,u_{k,i})(p_n(u_{2,i},\ldots,u_{k,i}))^{n-\ell}.\eeqn
Thus, for each $1\leq i \leq T$, $(u_{2,i},\ldots,u_{k,i}) \in \F_q^{k-1}$ is a zero of the polynomial 
  \[g_{\ell}(x_2,\ldots,x_k):= p_{n-\ell}(x_2,\ldots,x_k)(p_n(x_2,\ldots,x_k))^{n-1}-r_{n-\ell}(x_2,\ldots,x_k)(p_n(x_2,\ldots,x_k))^{n-\ell}.\]
Since $p_{n}$ and $ p_{n-\ell}$ have degree at most $d$, and $r_{n-\ell}$ has degree at most $\ell d$, it follows that $g_{\ell}$ has degree at most $nd$. In view of the Schwartz-Zippel lemma, to conclude that $|Z|=T\leq ndq^{k-2}$, it suffices to prove the following claim.

\begin{claim}
    There exists $\ell_0$ with $1\leq \ell_0\leq n$ so that $g_{\ell_0}(x_2,\ldots,x_k)$ is not identically zero.
\end{claim}

\begin{poc}
We give a proof by contradiction. Assume that $g_{\ell}(x_2,\ldots,x_k)$ is identically zero for any $\ell\in \{1,\ldots, n\}$.
   Then it follows from equation~\eqref{eq:fph2} that 
$$(p_n(x_2,\ldots,x_k))^{n-1}f(x_1,x_2,\ldots,x_k)=h^2(x_1,x_2,\ldots,x_k).$$
   We can write $f(x_1,x_2,\ldots,x_k)=f_1(x_1,x_2,\ldots,x_k)(f_2(x_1,x_2,\ldots,x_k))^2$ such that $f_1(x_1,x_2,\ldots,x_k)$ is square-free. Note that $f_1$ still satisfies the primitive condition. Then we have \[(p_n(x_2,\ldots,x_k))^{n-1}f_1(x_1,x_2,\ldots,x_k)=(h(x_1,x_2,\ldots,x_k)/f_2(x_1,x_2,\ldots,x_k))^2.\] Since $f_1$ is square-free, we must have $f_1(x_1,x_2,\ldots,x_k) \mid (p_n(x_2,\ldots,x_k))^{n-1}$. Thus, \[f_1(x_1,x_2,\ldots,x_k)=g(x_2,\ldots,x_k)\] for some polynomial $g$. By the primitive condition, $g$ has no zero over $\overline{\F_q}$. This forces $g$ to be a constant polynomial and thus $f(x_1,x_2,\ldots,x_k)=C(f_2(x_1,x_2,\ldots,x_k))^2$ for some $C\in \F_q$, violating our assumption. 
\end{poc}
\end{proof}

\begin{remark}
    When $f(x_1,x_2, \ldots, x_k)\in \F_q[x_1,x_2,\ldots, x_k]$ is diagonal and symmetric with even degree $d$, that is,  $f(x_1,x_2, \ldots, x_k)=ax_1^{d}+ \cdots +ax_k^{d}$ with $a\in 
\F_q^*$, well-known estimates on the number of solutions of the diagonal equations can give bounds on $|X|$. In this case, we have
\begin{align*} 
X
&=\{(u_2, \ldots, u_k)\in \F_q^{k-1}~|~ au_2^{d}+ \cdots + au_k^{d}=0\}.
\end{align*}
For instance, \cite[Theorem 6.36]{LN97} gives
$|X|=q^{k-2}+O_d(q^{k/2}).$
We can observe that the bound given in \cref{prop:Y} is sharp up to the constant factor $d^2+d$.
\end{remark}

\section{Proof of \cref{thm:main}} \label{Sec4}
\subsection{A sufficient condition}
Now we provide the following sufficient condition on the quasirandomness of $Y_{f,q}$.
\begin{proposition}\label{prop:key}
Let $f\in \F_q[x_1,x_2,\ldots , x_{k}]$ be a symmetric polynomial of degree $d$. Let $\mathcal{B}$ be the set of tuples $(u_2(0), u_2(1), \ldots, u_k(0), u_k(1))\in \F_q^{2k-2}$ such that
$$
\prod_{\substack{\varepsilon_i \in \{0,1\}\\2\leq i \leq k}} f\big(x, u_2(\varepsilon_2), u_3(\varepsilon_3), \ldots, u_k(\varepsilon_k)\big)
$$
is a constant multiple of the square of a polynomial in $\F_q[x]$. Then 
$$N_{Y_{f,q}}^*(\mathcal{O}_k^e)= \frac{q^{2k}}{2}+O_d(q^{2k-1}+|\mathcal{B}|q^2).$$
\end{proposition}

Our proof relies on the celebrated Weil bound for complete character sums; see, for example, \cite[Theorem 5.41]{LN97}.

\begin{lemma}[Weil's bound]\label{Weil}
Let $\chi$ be a multiplicative character of $\F_q$ of order $e>1$, and let $g \in \F_q[x]$ be a monic polynomial of positive degree that is not an $e$-th power of a polynomial. 
Let $s$ be the number of distinct roots of $g$ in its
splitting field over $\F_q$. Then for any $a \in \F_q$,
$$\bigg |\sum_{x\in\mathbb{F}_q}\chi\big(ag(x)\big) \bigg|\le(s-1)\sqrt q.$$
\end{lemma}

\begin{proof}[Proof of \textup{\cref{prop:key}}]
Let $\chi$ be the quadratic character of $\F_q$. Let $v_1,v_2,\ldots, v_k$ be distinct vertices of $Y_{f,q}$. Then by definition, $\{v_1,v_2,\ldots, v_k\}$ forms an edge in $Y_{f,q}$ if and only if $\chi(f(v_1,v_2,\ldots, v_k))\in \{0,1\}$. Thus, $2k$ distinct vertices $u_1(0),u_1(1), \ldots, u_k(0),u_k(1)$ forms an induced copy of an EPO if and only if there are an even number of vectors $\mathcal{E}=(\varepsilon_1, \varepsilon_2, \ldots, \varepsilon_k)\in \{0,1\}^k$ such that
$$
\chi\big(f(u_1(\varepsilon_1),u_2(\varepsilon_2),\ldots, u_k(\varepsilon_k))\big)=-1.
$$
We then consider two cases:

\smallskip
\textbf{Case 1:} $f(u_1(\varepsilon_1),u_2(\varepsilon_2),\ldots, u_k(\varepsilon_k))=0$ for some $\mathcal{E} \in \{0,1\}^k$. By the Schwartz-Zippel lemma, the number of such induced copies of EPOs is at most $2^k d q^{2k-1}$.

\smallskip
\textbf{Case 2:} $f(u_1(\varepsilon_1),u_2(\varepsilon_2),\ldots, u_k(\varepsilon_k))\neq 0$ for all $\mathcal{E} \in \{0,1\}^k$.
In this case, 
$$
\prod_{\mathcal{E} \in \{0,1\}^k}\chi\big(f(u_1(\varepsilon_1),u_2(\varepsilon_2),\ldots, u_k(\varepsilon_k)))\big)=(-1)^r,
$$
where $r$ is the number of non-edges. Thus, we have
\begin{small}
\begin{align*}
&\frac{1}{2} \bigg(1+\prod_{\mathcal{E} \in \{0,1\}^k}\chi\big(f(u_1(\varepsilon_1),u_2(\varepsilon_2),\ldots, u_k(\varepsilon_k))\big)\bigg)\\
&=\begin{cases}
    1, & \text{if } u_1(0),u_1(1), \ldots, u_k(0),u_k(1) \text{ forms an induced copy of an EPO}\\
    0, & \text{otherwise}.
\end{cases}
\end{align*}
\end{small}
Therefore, we deduce that
$$ N_{Y_{f,q}}^*(\mathcal{O}_k^e)=\sum\mathop{'}\frac{1}{2}\bigg(1+\prod_{\mathcal{E} \in \{0,1\}^k}\chi\big(f(u_1(\varepsilon_1),u_2(\varepsilon_2),\ldots, u_k(\varepsilon_k))\big)\bigg)+O_d(q^{2k-1}),$$
where the primed sum runs over all $2k$-tuples $(u_1(0),u_1(1) 
,\ldots, u_k(0),u_k(1)) \in \F_q^{2k}$ with distinct entries. Note that the number of $2k$-tuples $(u_1(0),u_1(1) 
,\ldots, u_k(0),u_k(1)) \in \F_q^{2k}$ with repeated entries is at most $\binom{2k}{2}q^{2k-1}$. Therefore, we still have
$$
N_{Y_{f,q}}^*(\mathcal{O}_k^e)=\sum\frac{1}{2}\bigg(1+\prod_{\mathcal{E} \in \{0,1\}^k}\chi\big(f(u_1(\varepsilon_1),u_2(\varepsilon_2),\ldots, u_k(\varepsilon_k))\big)\bigg)+O_d(q^{2k-1}),
$$
where the sum runs over all $2k$-tuples $(u_1(0),u_1(1) 
,\ldots, u_k(0),u_k(1)) \in \F_q^{2k}$. It follows that $N_{Y_{f,q}}^*(\mathcal{O}_k^e)=\frac{q^{2k}}{2}+S+O_d(q^{2k-1})$, 
where $S$ is the following sum below running over all $2k$-tuples $(u_1(0),u_1(1) 
,\ldots, u_k(0),u_k(1)) \in \F_q^{2k}$: 
\begin{align*}
    S:=\sum\prod_{\mathcal{E} \in \{0,1\}^k}\chi\big(f(u_1(\varepsilon_1),u_2(\varepsilon_2),\ldots, u_k(\varepsilon_k))\big). 
\end{align*}

By changing the order of summation, we have
\begin{footnotesize}
    \begin{align*}
    S=& \sum_{\substack{u_2(0),u_2(1) \\ 
,\ldots, \\ u_k(0),u_k(1) }} \sum_{u_1(0)}\sum_{u_1(1)}\left(\prod_{\substack{\varepsilon_1=0 \\ \mathcal{E} \in \{0,1\}^k}}\chi\big(f(u_1(\varepsilon_1),u_2(\varepsilon_2),\ldots, u_k(\varepsilon_k))\big)\prod_{\substack{\varepsilon_1=1 \\ \mathcal{E} \in \{0,1\}^k}} \chi\big(f(u_1(\varepsilon_1),u_2(\varepsilon_2),\ldots, u_k(\varepsilon_k))\big)\right)\\
=&\sum_{\substack{u_2(0),u_2(1) \\ 
,\ldots, \\ u_k(0),u_k(1) }} \underbrace{\left(\sum_{u_1(0)}\chi\bigg(\prod_{\substack{\varepsilon_1=0 \\ \mathcal{E} \in \{0,1\}^k}} f\big(u_1(0),u_2(\varepsilon_2),\ldots, u_k(\varepsilon_k)\big)\bigg)\right)}_{(1)}\underbrace{\left(\sum_{u_1(1)}\chi\bigg(\prod_{\substack{\varepsilon_1=1 \\ \mathcal{E} \in \{0,1\}^k}} f\big(u_1(1),u_2(\varepsilon_2),\ldots, u_k(\varepsilon_k)\big)\bigg)\right)}_{(2)}.
\end{align*}
\end{footnotesize}

Note that, if we fix $(u_2(0),u_2(1),\ldots,u_k(0),u_k(1))\in \F_q^{2k-2}$, then 
\[g(x):=\prod_{\substack{\varepsilon_1=0 \\ \mathcal{E} \in \{0,1\}^k}} f(x,u_2(\varepsilon_2),\ldots, u_k(\varepsilon_k))\]
is a polynomial in $x$ of degree at most $2^{k-1}d$ ( also when $\epsilon_1=1$). Thus, the absolute value of both $(1)$ and $(2)$ is at most $q$ provided that $g(x)$ is a constant multiple of the square of a polynomial in $\F_q[x]$; otherwise, Weil's bound implies that the absolute value of both $(1)$ and $(2)$ is at most $2^{k-1}d\sqrt{q}$. Thus, we can bound $|S|$ by splitting the sum over $\mathcal{B}$ and $\F_q^{2k-2} \setminus \mathcal{B}$ to conclude that
$$
|S|\leq |\mathcal{B}|q^2+(q^{2k-2}-|\mathcal{B}|)(2^{k-1}d\sqrt{q})^2\leq |\mathcal{B}|q^2+q^{2k-2}(2^{k-1}d\sqrt{q})^2=O_d(q^{2k-1}+|\mathcal{B}|q^2),
$$
as required.
\end{proof}

\subsection{Proof of \cref{thm:main} via \cref{prop:key}}

Let $Y_{f,q} \in \mathcal{Y}_d$; we claim that
$$
N^*_{Y_{f,q}}(\mathcal{O}_k^e)=\frac{q^{2k}}{2}+O_d(q^{2k-1}).
$$
As in \cref{prop:key}, let $\mathcal{B}$ be the set of tuples $(u_2(0), u_2(1), \ldots, u_k(0), u_k(1))\in \F_q^{2k-2}$ such that
\begin{equation}\label{prod of f}
\prod_{\substack{\varepsilon_i \in \{0,1\}\\2\leq i \leq k}} f\big(x, u_2(\varepsilon_2), u_3(\varepsilon_3), \ldots, u_k(\varepsilon_k)\big)
\end{equation}
is a constant multiple of the square of a polynomial in $\F_q[x]$.
In view of \cref{prop:key}, to complete the proof  of \cref{thm:main}, it suffices to show that $|\mathcal{B}|=O_d(q^{2k-3}).$
Since $f$ satisfies the primitive condition, by definition, the polynomial $f(x, v_2, v_3, \ldots, v_k)$ is not identically zero for any $(v_2,v_3,\ldots, v_k)\in \overline{\F_q}^{k-1}$. This fact will be used repeatedly throughout the proof below. In particular, this fact implies that the polynomial in \cref{prod of f} is not identically zero for any $(u_2(0), u_2(1), \ldots, u_k(0), u_k(1))\in \F_q^{2k-2}$.

We first give an upper bound on the number of tuples $(u_2(0), u_2(1), \ldots, u_k(0), u_k(1))\in \mathcal{B}$ such that $f(x, u_2(0),u_3(0),\ldots,u_k(0))$ is a constant multiple of the square of a polynomial in $\F_q[x]$. By \cref{prop:Y}, the number of such tuples is at most 
$O_d(q^{k-2})$, which is hence upper-bounded by $O_d(q^{2k-3})$.
It remains to bound the number of tuples $(u_2(0), u_2(1), \ldots, u_k(0), u_k(1))$ in $\mathcal{B}$ such that $f(x, u_2(0),u_3(0),\ldots,u_k(0))$ is not a constant multiple of the square of a polynomial in $\F_q[x]$. Note that for such a tuple in $\mathcal{B}$, since the polynomial in \cref{prod of f} is not identically zero, there exist $\varepsilon_2, \ldots, \varepsilon_k \in \{0,1\}$ that are not all zero, such that     
$$\gcd\big(f(x,u_2(0),\ldots, u_k(0)),f(x, u_2(\varepsilon_2), u_3(\varepsilon_3), \ldots, u_k(\varepsilon_k))\big)$$ has degree at least $1$; equivalently, these two polynomials have a common root $c\in \overline{\F_q}$. Note that for each $(u_2(0), u_3(0),\ldots, u_k(0))\in \F_q^{k-1}$, the polynomial $f(x,u_2(0),\ldots, u_k(0))$ has at most $d$ roots in $\overline{\F_q}$ since it is not identically zero. Thus, it suffices to show that if $u_2(0),\ldots, u_k(0) \in \F_q$ are fixed, and $c \in \overline{\F_q}$ is fixed, then for each $\varepsilon_2, \ldots, \varepsilon_k \in \{0,1\}$ not all zero, the number of tuples $(u_2(1), u_3(1),\ldots, u_k(1))\in \F_q^{k-1}$ such that $f(c, u_2(\varepsilon_2), u_3(\varepsilon_3), \ldots, u_k(\varepsilon_k))=0$
is $O_d(q^{k-2})$. 

Let $u_2(0),u_3(0),\ldots, u_k(0) \in \F_q$, and $c \in \overline{\F_q}$ be all fixed. Let  $\varepsilon_2, \ldots, \varepsilon_k \in \{0,1\}$ such that they are not all zero. 
Since $f$ is symmetric, without loss of generality, we can assume that there is $t\geq 2$, such that $\varepsilon_2=\varepsilon_3=\cdots=\varepsilon_{t}=1$ and $\varepsilon_{t+1}=\varepsilon_{t+2}=\cdots=\varepsilon_{k}=0$. Then 
$$f(x_2,x_3, \ldots, x_t, c, u_{t+1}(0), u_{t+2}(0), \ldots, u_k(0))=f(c, x_2,x_3, \ldots, x_t, u_{t+1}(0), u_{t+2}(0), \ldots, u_k(0))$$
is a polynomial with $t-1$ variables defined over $\overline{\F_q}$, and it is not identically zero since $t\geq 2$. Thus, by the Schwartz-Zippel lemma, the number of tuples $(u_2(1), u_3(1),\ldots, u_t(1))\in \F_q^{t-1}$ such that
$f(c, u_2(1), u_3(1),\ldots, u_t(1), u_{t+1}(0), u_{t+2}(0), \ldots, u_k(0))=0$
is at most $dq^{t-2}$. Since $u_{t+1}(1), u_{t+2}(1), \ldots, u_k(1)$ can be chosen arbitrarily in $\F_q$, it follows that the number of tuples $(u_2(1), u_3(1),\ldots, u_k(1))\in \F_q^{k-1}$ such that 
$$f(c, u_2(\varepsilon_2), u_3(\varepsilon_3), \ldots, u_k(\varepsilon_k))=f(c, u_2(1), u_3(1),\ldots, u_t(1), u_{t+1}(0), u_{t+2}(0), \ldots, u_k(0))=0$$
 is at most $dq^{t-2} \cdot q^{k-t}=dq^{k-2}=O_d(q^{k-2})$, as required.

\begin{remark}\label{rem:moregeneral}  
\cref{thm:main} in fact applies to a slightly larger family of polynomials. Let $f$ be a symmetric polynomial of degree $d$ in $\F_q[x_1,x_2,\ldots,x_k]$  that is not necessarily admissible. Assume that $f$ can be written as $g \cdot h$, where $g$ is a symmetric polynomial that can be written as a product of polynomials depending at most $k-1$ variables, and $h$ is an admissible polynomial. For such an $f$, one can begin with the product in \cref{prod of f} and apply a similar argument to show the same result for $N^*_{Y_{f,q}}(\mathcal{O}_k^e)$.
\end{remark}

\section{Asymptotic formula for the number of $f$-Diophantine sets with size $m$}\label{Sec5}

Recently, Slavov \cite{S26} used an explicit version of the Lang-Weil bound~\cite[Theorem 7.1]{CM06} to prove a result on the joint distribution of square values of polynomials over a finite field. The following lemma is a special case of his result \cite[Theorem 3 and Remark 13]{S26}.

\begin{lemma}[Slavov]\label{lem:LW}
Let $n,m,d$ be positive integers and $q$ be an odd prime power. Let $f_1, f_2,\ldots, f_n\in\F_q[x_1,x_2,\ldots, x_m]$ with degree $d$. Then the following are equivalent:

(i) For any nonempty subset $T \subset \{1,2,\ldots, n\}$, the polynomial $\prod_{i \in T} f_i$ is not a constant multiple of the square of a polynomial in $\F_q[x_1,\ldots, x_m]$.

(ii) The number of $(a_1, a_2, \ldots, a_m)\in \F_q^m$ such that $f_i(a_1,a_2, \ldots, a_m)$ is a nonzero square in $\F_q$ for all $1\leq i \leq n$ is 
$$
\frac{q^m}{2^n}+O((2d)^{2n}q^{m-1/2}+(2d)^{13n/3}q^{m-1}),
$$
where the implied constant in the error term is absolute.
\end{lemma}

We end the paper by presenting a proof of \cref{thm:asymp}. Most arguments used in the proof below have already appeared in previous sections.

\begin{proof}[Proof of \cref{thm:asymp}]
Let $q$ be an odd prime power and $f \in \F_q[x_1,x_2,\ldots,x_k]$ be an admissible polynomial with degree $d$. For each $m\geq k$, let $N(f,m)$ be the number of $f$-Diophantine sets of size $m$. 

Let $\chi$ be the quadratic character of $\F_q$, and let $\Tilde{\chi}:\F_q \to \F_q$ such that $\Tilde{\chi}(x)=\chi(x)$ for $x\neq 0$, and $\Tilde{\chi}(0)=1$. We can express $N(f,m)$ in terms of the following character sum:
\[
N(f,m)
 =\sum_{\substack{A\subset \mathbb{F}_q \\ |A|=m }} \prod_{\substack{T=\{b_1,\ldots, b_k\}\subset A\\ |T|=k}}\bigg(\frac{1+\Tilde{\chi}\big(f(b_1,\ldots , b_k)\big)}{2}\bigg).
\]
Let
\begin{equation}\label{eq:SS}
S(f,m):=\sum\mathop{'} \prod_{\substack{I=\{i_1, i_2,\ldots, i_k\}\subset \{1,2, \ldots, m\}\\ |I|=k}}\bigg(\frac{1+\chi\big(f(a_{i_1},\ldots , a_{i_k})\big)}{2}\bigg),
\end{equation}
with the primed sum runs over all pairwise distinct $(a_1, a_2, \ldots, a_m) \in \F_q^m$ such that $f(a_{i_1},\ldots , a_{i_k})\neq 0$ for all $I=\{i_1, i_2,\ldots, i_k\}\subset \{1,2, \ldots, m\}$ with $|I|=k$. Similar to the proof of \cref{prop:key}, we have
\begin{equation}\label{eq:N}
N(f,m)=\frac{S(f,m)}{m!}+O(dq^{m-1}).
\end{equation}
Let $\mathcal{I}$ be the collection of all subsets of $\{1,2,\ldots, m\}$ of size $k$. For each $I=\{i_1, i_2,\ldots, i_k\} \in \mathcal{I}$, define $f_I \in \F_q[x_1,x_2, \ldots, x_m]$ by \[f_I(x_1,x_2, \ldots, x_m)=f(x_{i_1}, x_{i_2}, \ldots, x_{i_k}).\] Based on equation~\eqref{eq:SS}, we see that $S(f,m)$ precisely counts the number of pairwise distinct $(a_1, a_2, \ldots, a_m)\in \F_q^m$ such that $f_I(a_1,a_2, \ldots, a_m)$ is a nonzero square in $\F_q$ for all $I \in \mathcal{I}$. 

Next, we show that for each nonempty subset $T \subset \mathcal{I}$, the polynomial $\prod_{I \in T} f_I$ is not a constant multiple of the square of a polynomial in $\F_q[x_1,\ldots, x_m]$. When $|T|=1$, this is obvious from the definition of admissible polynomials. Suppose otherwise that there is a subset $T \subset \mathcal{I}$ with $|T|\geq 2$ such that the polynomial $\prod_{I \in T} f_I$ is a constant multiple of the square of a polynomial in $\F_q[x_1,\ldots, x_m]$. Pick $J \in T$. From the definition of admissible polynomials, $f_{J}$ is not a constant multiple of the square of a polynomial in $\F_q[x_1,\ldots, x_m]$. Thus, there is $J' \in T$ with $J' \neq J$, such that $\gcd(f_J, f_{J'})$ is a polynomial with degree at least $1$. It follows that $J \cap J' \neq \emptyset$. Without loss of generality, assume that $J \cap J'=\{1,2,\ldots, t\}$ with $1\leq t \leq k-1$. Then $\gcd(f_J, f_{J'})=g(x_1, x_2, \ldots, x_t)$ for some polynomial $g \in \F_q[x_1,x_2, \ldots, x_t]$ with degree at least $1$, and in particular $g$ is a factor of $f$. However, this violates the primitive condition of the definition of admissible polynomials. 

Thus, by setting $n=\binom{m}{k}$ applied to the collection of polynomials $\{f_I: I \in \mathcal{I}\}$, we have verified condition (i) in~\cref{lem:LW}, and thus (ii) in the lemma implies that
\begin{equation}\label{eq:M}
M(f,m)=\frac{q^m}{2^{\binom{m}{k}}}
+O\big((2d)^{2\binom{m}{k}}q^{m-1/2}
+(2d)^{13\binom{m}{k}/3}q^{m-1}\big),
\end{equation}
where $M(f,m)$ denotes the number of tuples $(a_1,\ldots,a_m)\in \F_q^m$
such that $f_I(a_1,\ldots,a_m)$ is a nonzero square in $\F_q$ for all
$I\in \mathcal{I}$.
Moreover, $M(f,m)-S(f,m)$ is bounded by the number of tuples $(a_1,\ldots,a_m)\in \F_q^m$ with $a_i=a_j$ for some
$i\neq j$, which is at most $\binom{m}{2}q^{m-1}$.
Thus,
\begin{equation}\label{eq:S}
S(f,m)=M(f,m)+O(m^2q^{m-1}).
\end{equation}
Combining equations~\eqref{eq:N},~\eqref{eq:M}, and~\eqref{eq:S}, we obtain the desired result.
\end{proof}
 \section*{Acknowledgments}
 The authors thank Jaehoon Kim and Joonkyung Lee for their valuable feedback on a preliminary version of the manuscript. S. Yoo thanks Hyunwoo Lee for helpful discussions on the quasi-randomness of hypergraphs. The authors are also grateful to anonymous referees for their valuable comments and corrections. The research of C.H. Yip was supported in part by an NSERC fellowship. S. Yoo was supported by the Institute for Basic Science (IBS-R029-C1). 

\bibliographystyle{abbrv}
\bibliography{references}

@misc{KYY24b,
       author = {{Kim},Seoyoung and {Yip}, Chi Hoi and {Yoo}, Semin},
        title = "{{P}aley-like quasi-random graphs arising from polynomials}",
          note = {arXiv:2405.09319, 2024},
}

@article {KN,
    AUTHOR = {Kazalicki, Matija and Naskrecki, Bartosz},
     TITLE = {Diophantine triples and {K}3 surfaces},
   JOURNAL = {J. Number Theory},
  FJOURNAL = {Journal of Number Theory},
    VOLUME = {236},
      YEAR = {2022},
     PAGES = {41--70},
      ISSN = {0022-314X,1096-1658},
   MRCLASS = {14H52 (11D09 11G05 14J28)},
  MRNUMBER = {4395340},
MRREVIEWER = {Bidisha\ Roy},
       DOI = {10.1016/j.jnt.2021.07.009},
       URL = {https://doi.org/10.1016/j.jnt.2021.07.009},
}

@article {GSS02,
    AUTHOR = {Gyarmati, K. and S\'ark\"ozy, A. and Stewart, C. L.},
     TITLE = {On shifted products which are powers},
   JOURNAL = {Mathematika},
  FJOURNAL = {Mathematika. A Journal of Pure and Applied Mathematics},
    VOLUME = {49},
      YEAR = {2002},
    NUMBER = {1-2},
     PAGES = {227--230},
      ISSN = {0025-5793},
   MRCLASS = {11D45},
  MRNUMBER = {2059056},
MRREVIEWER = {Andrej\ Dujella},
       DOI = {10.1112/S0025579300016193},
       URL = {https://doi.org/10.1112/S0025579300016193},
}

@article {HTZ19,
    AUTHOR = {He, Bo and Togb\'{e}, Alain and Ziegler, Volker},
     TITLE = {There is no {D}iophantine quintuple},
   JOURNAL = {Trans. Amer. Math. Soc.},
  FJOURNAL = {Transactions of the American Mathematical Society},
    VOLUME = {371},
      YEAR = {2019},
    NUMBER = {9},
     PAGES = {6665--6709},
      ISSN = {0002-9947},
   MRCLASS = {11B37 (11D09 11J68 11J86 11Y50)},
  MRNUMBER = {3937341},
MRREVIEWER = {Paul M. Voutier},
       DOI = {10.1090/tran/7573},
       URL = {https://doi.org/10.1090/tran/7573},
}

@article {CY26,
    AUTHOR = {Croot, Ernie and Yip, Chi Hoi},
     TITLE = {Diophantine tuples and product sets in shifted powers},
   JOURNAL = {J. Lond. Math. Soc. (2)},
  FJOURNAL = {Journal of the London Mathematical Society. Second Series},
    VOLUME = {113},
      YEAR = {2026},
    NUMBER = {3},
     PAGES = {Paper No. e70499},
      ISSN = {0024-6107,1469-7750},
   MRCLASS = {Prelim},
  MRNUMBER = {5042533},
       DOI = {10.1112/jlms.70499},
       URL = {https://doi.org/10.1112/jlms.70499},
}

@article {BHP25,
    AUTHOR = {Batta, G. and Hajdu, Lajos and Pongr\'acz, Andr\'as},
     TITLE = {On {D}iophantine graphs},
   JOURNAL = {J. Lond. Math. Soc. (2)},
  FJOURNAL = {Journal of the London Mathematical Society. Second Series},
    VOLUME = {111},
      YEAR = {2025},
    NUMBER = {5},
     PAGES = {Paper No. e70163, 21},
      ISSN = {0024-6107,1469-7750},
   MRCLASS = {11B75 (05C15 05C42 11D09 11D25)},
  MRNUMBER = {4911223},
       DOI = {10.1112/jlms.70163},
       URL = {https://doi.org/10.1112/jlms.70163},
}

@book {D24,
    AUTHOR = {Dujella, Andrej},
     TITLE = {{D}iophantine $m$-tuples and {E}lliptic Curves},
    SERIES = {Developments in Mathematics},
    VOLUME = {79},
 PUBLISHER = {Springer, Cham},
      YEAR = {2024},
       DOI = {https://doi.org/10.1007/978-3-031-56724-7},
}

@article {YY25,
  author = {{Yip}, Chi Hoi and {Yoo}, Semin},
        title = "{$F$-{D}iophantine sets over finite fields}",
   JOURNAL = {Int. J. Number Theory},
  FJOURNAL = {International Journal of Number Theory},
    VOLUME = {21},
      YEAR = {2025},
    NUMBER = {5},
     PAGES = {1043--1050},
}

@article {C90,
    AUTHOR = {Chung, Fan R. K.},
     TITLE = {Quasi-random classes of hypergraphs},
   JOURNAL = {Random Structures Algorithms},
  FJOURNAL = {Random Structures \& Algorithms},
    VOLUME = {1},
      YEAR = {1990},
    NUMBER = {4},
     PAGES = {363--382},
      ISSN = {1042-9832,1098-2418},
   MRCLASS = {05C65 (05C80)},
  MRNUMBER = {1138430},
MRREVIEWER = {A.\ G.\ Thomason},
       DOI = {10.1002/rsa.3240010401},
       URL = {https://doi.org/10.1002/rsa.3240010401},
}

@article {AS,
    AUTHOR = {Adolphson, A. and Sperber, Steven},
     TITLE = {Character sums in finite fields},
   JOURNAL = {Compositio Math.},
  FJOURNAL = {Compositio Mathematica},
    VOLUME = {52},
      YEAR = {1984},
    NUMBER = {3},
     PAGES = {325--354},
      ISSN = {0010-437X,1570-5846},
   MRCLASS = {11L10 (11T21 14F30 14G15)},
  MRNUMBER = {756726},
MRREVIEWER = {Daniel\ Barsky},
       URL = {http://www.numdam.org/item?id=CM_1984__52_3_325_0},
}

@article {21DK,
    AUTHOR = {Dujella, Andrej and Kazalicki, Matija},
     TITLE = {Diophantine {$m$}-tuples in finite fields and modular forms},
   JOURNAL = {Res. Number Theory},
  FJOURNAL = {Research in Number Theory},
    VOLUME = {7},
      YEAR = {2021},
    NUMBER = {1},
     PAGES = {Paper No. 3, 24},
}

@book {LN97,
    AUTHOR = {Lidl, Rudolf and Niederreiter, Harald},
     TITLE = {Finite fields},
    SERIES = {Encyclopedia of Mathematics and its Applications},
    VOLUME = {20},
   EDITION = {Second},
 PUBLISHER = {Cambridge University Press, Cambridge},
      YEAR = {1997},
     PAGES = {xiv+755},
      ISBN = {0-521-39231-4},
   MRCLASS = {11Txx},
  MRNUMBER = {1429394},
}

@article {GM20,
    AUTHOR = {G\"{u}lo\u{g}lu, Ahmet M. and Murty, M. Ram},
     TITLE = {The {P}aley graph conjecture and {D}iophantine {$m$}-tuples},
   JOURNAL = {J. Combin. Theory Ser. A},
  FJOURNAL = {Journal of Combinatorial Theory. Series A},
    VOLUME = {170},
      YEAR = {2020},
     PAGES = {105155, 9},
      ISSN = {0097-3165},
   MRCLASS = {11D72},
  MRNUMBER = {4016095},
MRREVIEWER = {B. Sury},
       DOI = {10.1016/j.jcta.2019.105155},
       URL = {https://doi.org/10.1016/j.jcta.2019.105155},
}

@article {CGW89,
    AUTHOR = {Chung, F. R. K. and Graham, R. L. and Wilson, R. M.},
     TITLE = {Quasi-random graphs},
   JOURNAL = {Combinatorica},
  FJOURNAL = {Combinatorica. An International Journal on Combinatorics and
              the Theory of Computing},
    VOLUME = {9},
      YEAR = {1989},
    NUMBER = {4},
     PAGES = {345--362},
      ISSN = {0209-9683},
   MRCLASS = {05C80},
  MRNUMBER = {1054011},
MRREVIEWER = {Zbigniew Palka},
       DOI = {10.1007/BF02125347},
       URL = {https://doi.org/10.1007/BF02125347},
}

@article {BDHL11,
    AUTHOR = {B\'{e}rczes, Attila and Dujella, Andrej and Hajdu, Lajos and Luca,
              Florian},
     TITLE = {On the size of sets whose elements have perfect power
              {$n$}-shifted products},
   JOURNAL = {Publ. Math. Debrecen},
  FJOURNAL = {Publicationes Mathematicae Debrecen},
    VOLUME = {79},
      YEAR = {2011},
    NUMBER = {3-4},
     PAGES = {325--339},
      ISSN = {0033-3883},
   MRCLASS = {11B75 (11D45)},
  MRNUMBER = {2907969},
MRREVIEWER = {Christian Elsholtz},
       DOI = {10.5486/PMD.2011.5155},
       URL = {https://doi.org/10.5486/PMD.2011.5155},
}

@article {G01,
    AUTHOR = {Gyarmati, Katalin},
     TITLE = {On a problem of {D}iophantus},
   JOURNAL = {Acta Arith.},
  FJOURNAL = {Acta Arithmetica},
    VOLUME = {97},
      YEAR = {2001},
    NUMBER = {1},
     PAGES = {53--65},
      ISSN = {0065-1036},
   MRCLASS = {11D41 (11B99 11L40 11N36)},
  MRNUMBER = {1819622},
MRREVIEWER = {Andrej Dujella},
       DOI = {10.4064/aa97-1-3},
       URL = {https://doi.org/10.4064/aa97-1-3},
}

@article {GS08,
    AUTHOR = {Gyarmati, K. and S\'{a}rk\"{o}zy, A.},
     TITLE = {Equations in finite fields with restricted solution sets. {I}.
              {C}haracter sums},
   JOURNAL = {Acta Math. Hungar.},
  FJOURNAL = {Acta Mathematica Hungarica},
    VOLUME = {118},
      YEAR = {2008},
    NUMBER = {1-2},
     PAGES = {129--148},
      ISSN = {0236-5294,1588-2632},
   MRCLASS = {11T24},
  MRNUMBER = {2378545},
MRREVIEWER = {Ben\ Joseph\ Green},
       DOI = {10.1007/s10474-007-6192-5},
       URL = {https://doi.org/10.1007/s10474-007-6192-5},
}

@article {BG04,
    AUTHOR = {Bugeaud, Yann and Gyarmati, Katalin},
     TITLE = {On generalizations of a problem of {D}iophantus},
   JOURNAL = {Illinois J. Math.},
  FJOURNAL = {Illinois Journal of Mathematics},
    VOLUME = {48},
      YEAR = {2004},
    NUMBER = {4},
     PAGES = {1105--1115},
      ISSN = {0019-2082,1945-6581},
   MRCLASS = {11D99 (11B75)},
  MRNUMBER = {2113668},
MRREVIEWER = {P.\ G.\ Walsh},
       URL = {http://projecteuclid.org/euclid.ijm/1258138502},
}

@article {BDHT16,
    AUTHOR = {B\'{e}rczes, Attila and Dujella, Andrej and Hajdu, Lajos and
              Tengely, Szabolcs},
     TITLE = {Finiteness results for {$F$}-{D}iophantine sets},
   JOURNAL = {Monatsh. Math.},
  FJOURNAL = {Monatshefte f\"{u}r Mathematik},
    VOLUME = {180},
      YEAR = {2016},
    NUMBER = {3},
     PAGES = {469--484},
      ISSN = {0026-9255,1436-5081},
   MRCLASS = {11D45 (11D61)},
  MRNUMBER = {3513216},
MRREVIEWER = {Florian\ Luca},
       DOI = {10.1007/s00605-015-0824-6},
       URL = {https://doi.org/10.1007/s00605-015-0824-6},
}

@misc{KYY,
       author = {{Kim},Seoyoung and {Yip}, Chi Hoi and {Yoo}, Semin},
        title = "{Multiplicative structure of shifted multiplicative subgroups and its applications to Diophantine tuples}",
          note = {Canad. J. Math., to appear. \url{https://doi.org/10.4153/S0008414X25000136}},
}

@article {LW54,
    AUTHOR = {Lang, Serge and Weil, Andr\'e},
     TITLE = {Number of points of varieties in finite fields},
   JOURNAL = {Amer. J. Math.},
  FJOURNAL = {American Journal of Mathematics},
    VOLUME = {76},
      YEAR = {1954},
     PAGES = {819--827},
      ISSN = {0002-9327,1080-6377},
   MRCLASS = {14.0X},
  MRNUMBER = {65218},
MRREVIEWER = {B.\ Segre},
       DOI = {10.2307/2372655},
       URL = {https://doi.org/10.2307/2372655},
}

@article {CM06,
    AUTHOR = {Cafure, Antonio and Matera, Guillermo},
     TITLE = {Improved explicit estimates on the number of solutions of
              equations over a finite field},
   JOURNAL = {Finite Fields Appl.},
  FJOURNAL = {Finite Fields and their Applications},
    VOLUME = {12},
      YEAR = {2006},
    NUMBER = {2},
     PAGES = {155--185},
      ISSN = {1071-5797,1090-2465},
   MRCLASS = {11G25 (11D79)},
  MRNUMBER = {2206396},
MRREVIEWER = {Ioulia\ N.\ Baoulina},
       DOI = {10.1016/j.ffa.2005.03.003},
       URL = {https://doi.org/10.1016/j.ffa.2005.03.003},
}

@article{S26,
       author = {{Slavov}, Kaloyan},
        title = "{Square values of several polynomials over a finite field}",
       JOURNAL = {Finite Fields Appl.},
  FJOURNAL = {Finite Fields and their Applications},
    VOLUME = {109},
      YEAR = {2026},
     PAGES = {Paper No. 102696},
}

@incollection {T87a,
    AUTHOR = {Thomason, Andrew},
     TITLE = {Random graphs, strongly regular graphs and pseudorandom
              graphs},
 BOOKTITLE = {Surveys in combinatorics 1987 ({N}ew {C}ross, 1987)},
    SERIES = {London Math. Soc. Lecture Note Ser.},
    VOLUME = {123},
     PAGES = {173--195},
 PUBLISHER = {Cambridge Univ. Press, Cambridge},
      YEAR = {1987},
      ISBN = {0-521-34805-6},
   MRCLASS = {05C80 (05C35 05C55)},
  MRNUMBER = {905280},
MRREVIEWER = {Colin\ J. H. McDiarmid},
}

@incollection {T87b,
    AUTHOR = {Thomason, Andrew},
     TITLE = {Pseudorandom graphs},
 BOOKTITLE = {Random graphs '85 ({P}ozna\'{n}, 1985)},
    SERIES = {North-Holland Math. Stud.},
    VOLUME = {144},
     PAGES = {307--331},
 PUBLISHER = {North-Holland, Amsterdam},
      YEAR = {1987},
      ISBN = {0-444-70265-2},
   MRCLASS = {05C80},
  MRNUMBER = {930498},
MRREVIEWER = {Edward\ R.\ Scheinerman},
}

@article {KYY24a,
   author = {{Kim},Seoyoung and {Yip}, Chi Hoi and {Yoo}, Semin},
        title = "{Explicit constructions of {D}iophantine tuples over finite fields}",
   JOURNAL = {Ramanujan J.},
  FJOURNAL = {Ramanujan Journal. An International Journal Devoted to the
              Areas of Mathematics Influenced by Ramanujan},
    VOLUME = {65},
      YEAR = {2024},
    NUMBER = {1},
     PAGES = {163--172},
}

@article {Y24,
    AUTHOR = {Yip, Chi Hoi},
     TITLE = {Multiplicatively reducible subsets of shifted perfect {$k$}th
              powers and bipartite {D}iophantine tuples},
   JOURNAL = {Acta Arith.},
  FJOURNAL = {Acta Arithmetica},
    VOLUME = {218},
      YEAR = {2025},
    NUMBER = {3},
     PAGES = {251--271},
      ISSN = {0065-1036},
   MRCLASS = {11D45 (11B30 11D72 11N36)},
  MRNUMBER = {4889098},
       DOI = {10.4064/aa240520-24-9},
       URL = {https://doi.org/10.4064/aa240520-24-9},
}

@article {CG90,
    AUTHOR = {Chung, F. R. K. and Graham, R. L.},
     TITLE = {Quasi-random hypergraphs},
   JOURNAL = {Random Structures Algorithms},
  FJOURNAL = {Random Structures \& Algorithms},
    VOLUME = {1},
      YEAR = {1990},
    NUMBER = {1},
     PAGES = {105--124},
      ISSN = {1042-9832,1098-2418},
   MRCLASS = {05C65 (05C80)},
  MRNUMBER = {1068494},
MRREVIEWER = {A.\ G.\ Thomason},
       DOI = {10.1002/rsa.3240010108},
       URL = {https://doi.org/10.1002/rsa.3240010108},
}

@article {k23,
    AUTHOR = {Hammonds, Trajan and Kim, Seoyoung and Miller, Steven J. and
              Nigam, Arjun and Onghai, Kyle and Saikia, Dishant and Sharma,
              Lalit M.},
     TITLE = {{$k$}-{D}iophantine {$m$}-tuples in finite fields},
   JOURNAL = {Int. J. Number Theory},
  FJOURNAL = {International Journal of Number Theory},
    VOLUME = {19},
      YEAR = {2023},
    NUMBER = {4},
     PAGES = {891--912},
      ISSN = {1793-0421,1793-7310},
   MRCLASS = {11D79 (11D45)},
  MRNUMBER = {4555391},
       DOI = {10.1142/S1793042123500458},
       URL = {https://doi.org/10.1142/S1793042123500458},
}

@article {S80,
    AUTHOR = {Schwartz, J. T.},
     TITLE = {Fast probabilistic algorithms for verification of polynomial
              identities},
   JOURNAL = {J. Assoc. Comput. Mach.},
  FJOURNAL = {Journal of the Association for Computing Machinery},
    VOLUME = {27},
      YEAR = {1980},
    NUMBER = {4},
     PAGES = {701--717},
      ISSN = {0004-5411,1557-735X},
   MRCLASS = {68C20 (03B35 12-04)},
  MRNUMBER = {594695},
       DOI = {10.1145/322217.322225},
       URL = {https://doi.org/10.1145/322217.322225},
}

@article {Shparlinski23,
    AUTHOR = {Shparlinski, Igor E.},
     TITLE = {On the number of {D}iophantine {$m$}-tuples in finite fields},
   JOURNAL = {Finite Fields Appl.},
  FJOURNAL = {Finite Fields and their Applications},
    VOLUME = {90},
      YEAR = {2023},
     PAGES = {Paper No. 102241, 7},
}

@incollection {HT89,
    AUTHOR = {Haviland, Julie and Thomason, Andrew},
     TITLE = {Pseudo-random hypergraphs},
      NOTE = {Graph theory and combinatorics (Cambridge, 1988)},
   JOURNAL = {Discrete Math.},
  FJOURNAL = {Discrete Mathematics},
    VOLUME = {75},
      YEAR = {1989},
    NUMBER = {1-3},
     PAGES = {255--278},
      ISSN = {0012-365X,1872-681X},
}

@article {AH20,
    AUTHOR = {Aigner-Horev, Elad and H\`an, Hi\cfudot ep},
     TITLE = {Linear quasi-randomness of subsets of abelian groups and
              hypergraphs},
   JOURNAL = {European J. Combin.},
  FJOURNAL = {European Journal of Combinatorics},
    VOLUME = {88},
      YEAR = {2020},
     PAGES = {103116, 16},
      ISSN = {0195-6698,1095-9971},
}

@article {CHPS12,
    AUTHOR = {Conlon, David and H\`an, Hi\^ep and Person, Yury and Schacht,
              Mathias},
     TITLE = {Weak quasi-randomness for uniform hypergraphs},
   JOURNAL = {Random Structures Algorithms},
  FJOURNAL = {Random Structures \& Algorithms},
    VOLUME = {40},
      YEAR = {2012},
    NUMBER = {1},
     PAGES = {1--38},
}

@article {Gu26,
    AUTHOR = {Gu, Zijie},
     TITLE = {A generalisation of diophantine tuples},
   JOURNAL = {Ramanujan J.},
  FJOURNAL = {Ramanujan Journal. An International Journal Devoted to the
              Areas of Mathematics Influenced by Ramanujan},
    VOLUME = {69},
      YEAR = {2026},
    NUMBER = {3},
     PAGES = {72},
}

\end{document}